\documentclass[a4paper,10pt]{article}

\usepackage{amssymb,amsthm,amsmath,latexsym}

\usepackage{mathtools} 
\usepackage{authblk}
\usepackage[top=1in, bottom=1in, left=1in, right=1in]{geometry}
\usepackage{xcolor}

\newtheorem{theorem}{Theorem}[section]

\newtheorem{definition}[theorem]{Definition}

\begin{document}

\title{On a new definition of fractional differintegrals with Mittag-Leffler kernel}

\author[1]{Arran Fernandez \thanks{Email: \texttt{af454@cam.ac.uk}}}
\author[2,3]{Dumitru Baleanu \thanks{Email: \texttt{dumitru@cankaya.edu.tr}}}
\affil[1]{\small Department of Applied Mathematics and Theoretical Physics, University of Cambridge, Wilberforce Road, CB3 0WA, UK}
\affil[2]{\small Department of Mathematics, Cankaya University, 06530 Balgat, Ankara, Turkey}
\affil[3]{\small Institute of Space Sciences, Magurele-Bucharest, Romania}

\maketitle

\begin{abstract}
We introduce and study the properties of a new family of fractional differential and integral operators which are based directly on an iteration process and therefore satisfy a semigroup property. We also solve some ODEs in this new model and discuss applications of our results.
\end{abstract}

\section{Background and motivation}

Fractional calculus -- the study of differentiation and integration to non-integer orders -- is a branch of mathematics which has undergone rapid expansion in the last few decades thanks to the discovery of applications in many fields of science \cite{hilfer, kilbas, podlubny, samko}.

Researchers are trying to find the best families of fractional operators in order to better describe the complexity of various real-world phenomena. For an excellent review about the main achievements of fractional calculus up to the year of 1974, we recommend the reader to \cite{ross}, while for a review of progress made since then, the details may be found in \cite{baleanu-et-al, uchaikin}.

The papers of Liouville \cite{liouville} and Caputo \cite{caputo}, introducing the models of fractional calculus which are now standard and known as the Riemann--Liouville and Caputo definitions, were both motivated by real-world considerations. Each author created new fractional derivatives, which had not previously been used, based on the fact that they could be used successfully to model certain real problems.

The same process is ongoing even up to the present day, with papers such as \cite{kilbas2}, \cite{srivastava}, \cite{caputo2}, \cite{katugampola}, and \cite{atangana} appearing each year which introduce new models of fractional calculus in order to apply them in the real world. Fractional models defined using non-singular kernels were motivated by the existence of certain non-local systems, which have real-world applications in describing heterogeneities and fluctuations but which are not amenable to being modelled either by classical local calculus or by fractional calculus with singular kernels.

Thus an important question is how to classify these diverse definitions of fractional operators. But despite many attempts to define what makes an operator a fractional derivative, e.g. in \cite{ross} and \cite{ortigueira} (see also the correction \cite{katugampola2} to the latter), this important problem still remains open.

However, we think that there are more families of fractional derivatives and integrals, and that all together may describe related parts of the dynamics of nonlocal complex systems. The ultimate criteria of which fractional derivative is more suitable for a given real world process will be given by the relevant experimental data.

One interesting issue in fractional calculus is the way various formulations of fractional derivatives and integrals introduced the gamma function. In Liouville's approach \cite{liouville}, for example, it appears naturally: the integral formula for $1/x$ leads directly to the gamma function. In the approach of fractional integrals and derivatives with non-singular kernel \cite{atangana}, the gamma function appears naturally by the definition of the Mittag-Leffler function as a series.

We recall from \cite{atangana} that the formulae for fractional derivatives and integrals with Mittag-Leffler kernel are given by the equations \eqref{ABI:defn}-\eqref{ABC:defn}. This model of fractional calculus is usually referred to as the \textbf{AB model}. AB integrals are denoted by $\prescript{AB}{}I$ and defined by
\begin{equation}
\label{ABI:defn}
\prescript{AB}{}I^{\alpha}_{a+}f(t)=\frac{1-\alpha}{B(\alpha)}f(t)+\frac{\alpha}{B(\alpha)}\prescript{RL}{}I_{a+}^{\alpha}f(t).
\end{equation}
There are two distinct expressions for AB derivatives, according to whether the differentiation is done after or before the integration with kernel. These are called derivatives of Riemann--Liouville type and Caputo type, in analogy with the definitions of the classical Riemann--Liouville and Caputo fractional derivatives, and denoted by $\prescript{ABR}{}D$ and $\prescript{ABC}{}D$ respectively; their definitions are as follows.
\begin{align}
\prescript{ABR}{}D^{\alpha}_{a+}f(t)&=\frac{B(\alpha)}{1-\alpha}\frac{\mathrm{d}}{\mathrm{d}t}\int_a^tf(x)E_\alpha\Big(\tfrac{-\alpha}{1-\alpha}(t-x)^{\alpha}\Big)\,\mathrm{d}x; \label{ABR:defn}\\
\prescript{ABC}{}D^{\alpha}_{a+}f(t)&=\frac{B(\alpha)}{1-\alpha}\int_a^tf'(x)E_\alpha\Big(\tfrac{-\alpha}{1-\alpha}(t-x)^{\alpha}\Big)\,\mathrm{d}x. \label{ABC:defn}
\end{align}
Each of these formulae \eqref{ABI:defn}-\eqref{ABC:defn} is valid for $0<\alpha<1$, $a<b$ in $\mathbb{R}$, $f\in L^1[a,b]$, and the function $B$ is a multiplier which satisfies $B(0)=B(1)=1$. For simplicity, as was done in \cite{baleanu-fernandez}, we shall also assume that $B$ only takes real positive values. Applications of the AB model of fractional calculus have been explored in many recent papers, for example \cite{djida,alqahtani,gomezaguilar,owolabi}.

In the current work, we go beyond this to develop a new model of fractional calculus which, unlike the AB model \cite{baleanu-fernandez}, has a semigroup property. This is significant because the semigroup property is an intuitive criterion which it seems natural for fractional differintegrals to satisfy \cite{samko}. The starting point for our work, motivated by the recent paper \cite{jarad}, is to consider iterations of the AB formula, which can be used to derive a new expression for fractional differintegrals.

Our paper is structured as follows. In section 2 we derive the definition of our new fractional calculus, and verify that it reduces as expected for some simple values of the order of differintegration. In section 3 we prove some fundamental properties of these differintegrals, including that they are bounded operators and satisfy a semigroup property. In section 4 we solve some fractional ODEs in the new model and consider applications of our results.

\section{Deriving the formula}

The expression \eqref{ABI:defn} for AB fractional integrals can be rewritten in distributional form as follows:
\begin{align*}
\prescript{AB}{}I^{\alpha}_{a+}f(t)&=\frac{1-\alpha}{B(\alpha)}f(t)+\frac{\alpha}{B(\alpha)\Gamma(\alpha)}\int_{a}^t(t-x)^{\alpha-1}f(x)\,\mathrm{d}x \\
&=\int_{a}^tf(x)\left[\frac{1-\alpha}{B(\alpha)}\delta(t-x)+\frac{\alpha}{B(\alpha)\Gamma(\alpha)}(t-x)^{\alpha-1}\right]\,\mathrm{d}x,
\end{align*}
where $\delta$ is the Dirac delta function.

Iterating the AB integral an arbitrary natural number of times gives the following formula for sequential AB fractional integrals:
\begin{align}
\nonumber \left(\prescript{AB}{}I^{\alpha}_{a+}\right)^nf(t)&=\left[\frac{1-\alpha}{B(\alpha)}+\frac{\alpha}{B(\alpha)}\prescript{RL}{}I_{a+}^{\alpha}\right]^nf(t) \\
&=\sum_{k=0}^n\frac{\binom{n}{k}(1-\alpha)^{n-k}\alpha^k}{B(\alpha)^n}\prescript{RL}{}I_{a+}^{\alpha k}f(t) \label{iterate:n:RL}\\
&=\left(\frac{1-\alpha}{B(\alpha)}\right)^nf(t)+\sum_{k=1}^n\frac{\binom{n}{k}(1-\alpha)^{n-k}\alpha^k}{B(\alpha)^n\Gamma(k\alpha)}\int_{a}^t(t-x)^{k\alpha-1}f(x)\,\mathrm{d}x, \label{iterate:n:integral}
\end{align}
where we have used the fact that Riemann--Liouville integrals satisfy the semigroup property. This formula too can be written in distributional form, as follows:
\begin{equation}
\label{iterate:n:delta}
\left(\prescript{AB}{}I^{\alpha}_{a+}\right)^nf(t)=\int_{a}^tf(x)\left[\left(\frac{1-\alpha}{B(\alpha)}\right)^n\delta(t-x)+\sum_{k=1}^n\frac{\binom{n}{k}(1-\alpha)^{n-k}\alpha^k}{B(\alpha)^n\Gamma(k\alpha)}(t-x)^{k\alpha-1}\right]\,\mathrm{d}x.
\end{equation}

The series in equations \eqref{iterate:n:RL}-\eqref{iterate:n:delta} is a finite binomial series arising from the $n$th power. Thus it is easy to generalise to arbitrary powers, using an infinite binomial series. We define the $\beta$th iteration of the $\alpha$th AB integral, for $0<\alpha<1$ and $\beta\in\mathbb{R}$, by the following equivalent formulae. (We include all three of these formulae because each of them can be more useful than the others in particular contexts.)
\begin{align}
\left(\prescript{AB}{}I^{\alpha}_{a+}\right)^{\beta}f(t)&=\sum_{k=0}^{\infty}\frac{\binom{\beta}{k}(1-\alpha)^{\beta-k}\alpha^k}{B(\alpha)^{\beta}}\prescript{RL}{}I_{a+}^{\alpha k}f(t) \label{iterate:beta:RL}\\
&=\left(\frac{1-\alpha}{B(\alpha)}\right)^{\beta}f(t)+\sum_{k=1}^{\infty}\frac{\binom{\beta}{k}(1-\alpha)^{\beta-k}\alpha^k}{B(\alpha)^{\beta}\Gamma(k\alpha)}\int_{a}^t(t-x)^{k\alpha-1}f(x)\,\mathrm{d}x \label{iterate:beta:integral}\\
&=\int_{a}^t\left(\frac{1-\alpha}{B(\alpha)}\right)^{\beta}f(x)\left[\delta(t-x)+\sum_{k=1}^{\infty}\frac{\binom{\beta}{k}(1-\alpha)^{-k}\alpha^k}{\Gamma(k\alpha)}(t-x)^{k\alpha-1}\right]\,\mathrm{d}x. \label{iterate:beta:delta}
\end{align}
Note that this formula is valid regardless of the sign of $\beta$: it is a true fractional differintegral, covering both derivatives and integrals equally. We shall see below that this differintegral has various desirable properties, but first of all we formalise the definition as follows.

\begin{definition}
\label{IAB:defn}
Let $0\leq\alpha\leq1$, $\beta\in\mathbb{R}$, $a<b$ in $\mathbb{R}$, and $f:[a,b]\rightarrow\mathbb{R}$ be an $L^1$ function. The $\beta$th iteration of the $\alpha$th AB integral of a function $f$, which we shall call an \textbf{iterated AB differintegral} and denote by $\mathcal{I}_{a+}^{(\alpha,\beta)}f(t)$, is defined by the formulae \eqref{iterate:beta:RL}-\eqref{iterate:beta:delta}. In other words, the iterated AB integral is given by
\begin{equation}
\label{IAB:defn:int}
\mathcal{I}_{a+}^{(\alpha,\beta)}f(t)=\sum_{k=0}^{\infty}\frac{\binom{\beta}{k}(1-\alpha)^{\beta-k}\alpha^k}{B(\alpha)^{\beta}}\prescript{RL}{}I_{a+}^{\alpha k}f(t),
\end{equation}
and the iterated AB derivative is given by
\begin{equation}
\label{IAB:defn:deriv}
\mathcal{D}_{a+}^{(\alpha,\beta)}f(t)=\sum_{k=0}^{\infty}\frac{\binom{-\beta}{k}\alpha^kB(\alpha)^{\beta}}{(1-\alpha)^{\beta+k}}\prescript{RL}{}I_{a+}^{\alpha k}f(t).
\end{equation}
\end{definition}

In order to demonstrate the appropriateness of Definition \ref{IAB:defn}, let us consider how this differintegral behaves in different specific cases of the variables $\alpha$ and $\beta$.

\begin{itemize}
\item If $\alpha=0$, then the operator is trivial: \[\mathcal{I}_{a+}^{(0,\beta)}f(t)=f(t).\]
\item If $\beta=0$, then the operator is trivial: \[\mathcal{I}_{a+}^{(\alpha,0)}f(t)=f(t).\]
\item If $\beta=n\in\mathbb{N}$, then the original formulae \eqref{iterate:n:RL}-\eqref{iterate:n:delta} for iterated AB integrals are recovered: \[\mathcal{I}_{a+}^{(\alpha,n)}f(t)=\left(\prescript{AB}{}I^{\alpha}_{a+}\right)^nf(t).\]
\item If $\beta=-1$, then the operator is the ABR derivative, because \eqref{iterate:beta:integral} becomes the series expression found in \cite{baleanu-fernandez}, while \eqref{iterate:beta:delta} is analogous to the distributional formulation of the ABC derivative used in \cite{fernandez-baleanu}: \[\mathcal{I}_{a+}^{(\alpha,-1)}f(t)=\prescript{ABR}{}D^{\alpha}_{a+}f(t).\]
\item If $\beta=-n$, $n\in\mathbb{N}$, then similarly the operator is the iterated ABR derivative: \[\mathcal{I}_{a+}^{(\alpha,-n)}f(t)=\left(\prescript{ABR}{}D^{\alpha}_{a+}\right)^nf(t).\]
\end{itemize}

We shall now prove some basic properties of our new definition. In particular, we note that convergence of the series \eqref{IAB:defn:int} and \eqref{IAB:defn:deriv} is given by the boundedness of the associated operators, proved in Theorem \ref{bounded} below.

\section{Fundamental properties}

The Laplace transforms of iterated AB differintegrals are easy to compute, and behave as we would expect them to given the definition. More precisely, we have the following theorem.

\begin{theorem}[Laplace transforms]
\label{Laplace}
If $\alpha$, $\beta$, $a$, $b$, and $f$ are as in Definition \ref{IAB:defn} and $f$ has a well-defined Laplace transform, then the Laplace transform of its iterated AB differintegral is given by
\begin{equation}
\label{Laplace:eqn}
\mathcal{L}\left(\mathcal{I}_{0+}^{(\alpha,\beta)}f(t)\right)=\left(\frac{1-\alpha}{B(\alpha)}+\frac{\alpha}{B(\alpha)}s^{-\alpha}\right)^{\beta}\hat{f}(s),
\end{equation}
where $\mathcal{L}$ and $\hat{ }$ both denote the Laplace transform.
\end{theorem}

\begin{proof}
This follows from the formula \eqref{iterate:beta:RL}, since we know what the Laplace transforms of Riemann--Liouville fractional operators look like. More explicitly, the proof runs as follows.
\begin{align*}
\mathcal{L}\left(\mathcal{I}_{0+}^{(\alpha,\beta)}f(t)\right)&=\mathcal{L}\left(\sum_{k=0}^{\infty}\frac{\binom{\beta}{k}(1-\alpha)^{\beta-k}\alpha^k}{B(\alpha)^{\beta}}\prescript{RL}{}I_{0+}^{\alpha k}f(t)\right) \\
&=\sum_{k=0}^{\infty}\frac{\binom{\beta}{k}(1-\alpha)^{\beta-k}\alpha^k}{B(\alpha)^{\beta}}\mathcal{L}\left(\prescript{RL}{}I_{0+}^{\alpha k}f(t)\right) \\
&=\sum_{k=0}^{\infty}\frac{\binom{\beta}{k}(1-\alpha)^{\beta-k}\alpha^k}{B(\alpha)^{\beta}}s^{-\alpha k}\hat{f}(s) \\
&=\sum_{k=0}^{\infty}\frac{\binom{\beta}{k}(1-\alpha)^{\beta-k}(\alpha s^{-\alpha})^k}{B(\alpha)^{\beta}}\hat{f}(s) \\
&=\left(\frac{1-\alpha}{B(\alpha)}+\frac{\alpha}{B(\alpha)}s^{-\alpha}\right)^{\beta}\hat{f}(s),
\end{align*}
where for the last step we used the binomial theorem again.
\end{proof}

A very important aspect to consider for any fractional differintegral is the semigroup property, i.e. the question of whether or not a differintegral of a differintegral is a differintegral of the expected order. In the Riemann--Liouville model, fractional integrals satisfy the semigroup property, but fractional derivatives do not except under special conditions \cite{samko}. In the AB model, neither derivatives nor integrals satisfy the semigroup property \cite{baleanu-fernandez}. It turns out that in our new model, there is a semigroup property for all differintegrals.

\begin{theorem}[Semigroup property]
\label{semigroup}
Iterated AB differintegrals have a semigroup property in $\beta$, i.e.
\begin{equation}
\label{semigroup:eqn}
\mathcal{I}_{a+}^{(\alpha,\beta)}\mathcal{I}_{a+}^{(\alpha,\gamma)}f(t)=\mathcal{I}_{a+}^{(\alpha,\beta+\gamma)}f(t)
\end{equation}
for all $\alpha\in[0,1]$, $\beta,\gamma\in\mathbb{R}$, and $a,f$ as in Definition \ref{IAB:defn}.
\end{theorem}

\begin{proof}
Once again, this is a consequence of the fact that our new model is derived from binomial expansions. We use the formula \eqref{iterate:beta:RL} and the fact that Riemann--Liouville integrals have a semigroup property:
\begin{align*}
\mathcal{I}_{a+}^{(\alpha,\beta)}\mathcal{I}_{a+}^{(\alpha,\gamma)}f(t)&=\sum_{k=0}^{\infty}\frac{\binom{\beta}{k}(1-\alpha)^{\beta-k}\alpha^k}{B(\alpha)^{\beta}}\prescript{RL}{}I_{a+}^{\alpha k}\left[\sum_{j=0}^{\infty}\frac{\binom{\gamma}{j}(1-\alpha)^{\gamma-j}\alpha^j}{B(\alpha)^{\gamma}}\prescript{RL}{}I_{a+}^{\alpha j}f(t)\right] \\
&=\sum_{k=0}^{\infty}\sum_{j=0}^{\infty}\frac{\binom{\beta}{k}\binom{\gamma}{j}(1-\alpha)^{(\beta+\gamma)-(k+j)}\alpha^{k+j}}{B(\alpha)^{\beta+\gamma}}\prescript{RL}{}I_{a+}^{\alpha(k+j)}f(t) \\
&=\sum_{m=0}^{\infty}\sum_{k=0}^{m}\frac{\binom{\beta}{k}\binom{\gamma}{m-k}(1-\alpha)^{(\beta+\gamma)-m}\alpha^m}{B(\alpha)^{\beta+\gamma}}\prescript{RL}{}I_{a+}^{\alpha m}f(t) \\
&=\sum_{m=0}^{\infty}\frac{\binom{\beta+\gamma}{m}(1-\alpha)^{(\beta+\gamma)-m}\alpha^m}{B(\alpha)^{\beta+\gamma}}\prescript{RL}{}I_{a+}^{\alpha m}f(t)=\mathcal{I}_{a+}^{(\alpha,\beta+\gamma)}f(t),
\end{align*}
where we have used the binomial identity $\sum_{k=0}^m\binom{\beta}{k}\binom{\gamma}{m-k}=\binom{\beta+\gamma}{m}$.
\end{proof}

We also show that all differintegral operators in the new model are bounded in the $L^1$ and $L^{\infty}$ norms.

\begin{theorem}[Bounded operators]
\label{bounded}
Let $a,b,\alpha,\beta$ be as in Definition \ref{IAB:defn}. There exists a positive constant $K$ such that for any $f\in L^1[a,b]$,
\begin{equation}
\label{bounded:L1}
||\mathcal{I}_{a+}^{(\alpha,\beta)}f||_1\leq K||f||_1,
\end{equation}
and, if we also assume $f$ is continuous,
\begin{equation}
\label{bounded:Linfinity}
||\mathcal{I}_{a+}^{(\alpha,\beta)}f||_{\infty}\leq K||f||_{\infty}.
\end{equation}
\end{theorem}

\begin{proof}
We use the formula \eqref{iterate:beta:integral} to find bounds on $\mathcal{I}_{a+}^{(\alpha,\beta)}f(t)$.

By the first mean value theorem for integrals, provided $f$ is continuous (and therefore bounded), we have \[\int_{a}^t(t-x)^{k\alpha-1}f(x)\,\mathrm{d}x=f(c)\int_a^t(t-x)^{k\alpha-1}\,\mathrm{d}x=f(c)\frac{(t-a)^{k\alpha}}{k\alpha}\] for some $c\in(a,t)$, and therefore \[\left|\int_{a}^t(t-x)^{k\alpha-1}f(x)\,\mathrm{d}x\right|\leq||f||_{\infty}\frac{(b-a)^{k\alpha}}{k\alpha}.\] Thus, the formula \eqref{iterate:beta:integral} gives
\begin{align*}
\left|\mathcal{I}_{a+}^{(\alpha,\beta)}f(t)\right|&=\left|\left(\frac{1-\alpha}{B(\alpha)}\right)^{\beta}f(t)+\sum_{k=1}^{\infty}\frac{\binom{\beta}{k}(1-\alpha)^{\beta-k}\alpha^k}{B(\alpha)^{\beta}\Gamma(k\alpha)}\int_{a}^t(t-x)^{k\alpha-1}f(x)\,\mathrm{d}x\right| \\
&\leq\left(\frac{1-\alpha}{B(\alpha)}\right)^{\beta}||f||_{\infty}+\sum_{k=1}^{\infty}\frac{\left|\binom{\beta}{k}\right|(1-\alpha)^{\beta-k}\alpha^k}{B(\alpha)^{\beta}\Gamma(k\alpha)}\left|\int_{a}^t(t-x)^{k\alpha-1}f(x)\,\mathrm{d}x\right| \\
&\leq\left[\left(\frac{1-\alpha}{B(\alpha)}\right)^{\beta}+\sum_{k=1}^{\infty}\frac{\left|\binom{\beta}{k}\right|(1-\alpha)^{\beta-k}\alpha^k(b-a)^{k\alpha}}{B(\alpha)^{\beta}\Gamma(k\alpha+1)}\right]||f||_{\infty} \\
&=\left[\left(\frac{1-\alpha}{B(\alpha)}\right)^{\beta}\sum_{k=0}^{\infty}\left|\binom{\beta}{k}\right|\left(\frac{\alpha(b-a)^{\alpha}}{1-\alpha}\right)^k\frac{1}{\Gamma(k\alpha+1)}\right]||f||_{\infty}.
\end{align*}
The term in square brackets depends only on $a$, $b$, $\alpha$, and $\beta$, so we have proved \eqref{bounded:Linfinity}.

By the second mean value theorem for integrals, we have \[\int_{a}^t(t-x)^{k\alpha-1}|f(x)|\,\mathrm{d}x=(t-a)^{k\alpha-1}\int_a^c|f(x)|\,\mathrm{d}x\] for some $c\in(a,t]$, and therefore \[\left|\int_{a}^t(t-x)^{k\alpha-1}f(x)\,\mathrm{d}x\right|\leq||f||_1(t-a)^{k\alpha-1}.\] Thus, the formula \eqref{iterate:beta:integral} gives
\begin{equation*}
\left|\mathcal{I}_{a+}^{(\alpha,\beta)}f(t)\right|\leq\left(\frac{1-\alpha}{B(\alpha)}\right)^{\beta}|f(t)|+\sum_{k=1}^{\infty}\frac{\left|\binom{\beta}{k}\right|(1-\alpha)^{\beta-k}\alpha^k}{B(\alpha)^{\beta}\Gamma(k\alpha)}||f||_1(t-a)^{k\alpha-1}.
\end{equation*}
Integrating this inequality with respect to $t$ yields
\begin{align*}
\int_a^b\left|\mathcal{I}_{a+}^{(\alpha,\beta)}f(t)\right|\,\mathrm{d}t\leq\left(\frac{1-\alpha}{B(\alpha)}\right)^{\beta}\int_a^b|f(t)|\,\mathrm{d}t+\sum_{k=1}^{\infty}\frac{\left|\binom{\beta}{k}\right|(1-\alpha)^{\beta-k}\alpha^k}{B(\alpha)^{\beta}\Gamma(k\alpha)}||f||_1\frac{(b-a)^{k\alpha}}{k\alpha},
\end{align*}
and therefore $||\mathcal{I}_{a+}^{(\alpha,\beta)}f||_1\leq K||f||_1$ with the constant $K$ being exactly the same as before.
\end{proof}

\section{Differential equations and applications}

We now consider certain classes of fractional ordinary differintegral equations which can be solved in the new model. For example, let us solve the following equation:
\begin{equation}
\label{ODE1}
\mathcal{I}_{0+}^{(\alpha,\beta)}f(t)=P+Qf(t)+R(f(t))^2,
\end{equation}
for fixed $\alpha,\beta,P,Q,R$, using a series solution method with the following ansatz:
\begin{equation}
\label{ansatz}
f(t)=\sum_{n=0}^{\infty}a_nt^{n\alpha}.
\end{equation}
When $f(t)$ is in this form, we use the formula \eqref{iterate:beta:RL} to evaluate $\mathcal{I}_{0+}^{(\alpha,\beta)}f(t)$, as follows:
\begin{align}
\nonumber \mathcal{I}_{0+}^{(\alpha,\beta)}f(t)&=\sum_{k=0}^{\infty}\frac{\binom{\beta}{k}(1-\alpha)^{\beta-k}\alpha^k}{B(\alpha)^{\beta}}\prescript{RL}{}I_{a+}^{\alpha k}\left(\sum_{l=0}^{\infty}a_lt^{l\alpha}\right) \\
\nonumber &=\sum_{k=0}^{\infty}\sum_{l=0}^{\infty}\frac{\binom{\beta}{k}(1-\alpha)^{\beta-k}\alpha^k}{B(\alpha)^{\beta}}a_l\frac{\Gamma(l\alpha+1)}{\Gamma((k+l)\alpha+1)}t^{(k+l)\alpha} \\
\label{ODE1:LHS} &=\sum_{m=0}^{\infty}\frac{t^{m\alpha}}{B(\alpha)^{\beta}\Gamma(m\alpha+1)}\sum_{k=0}^ma_{m-k}\binom{\beta}{k}(1-\alpha)^{\beta-k}\alpha^k\Gamma((m-k)\alpha+1).
\end{align}
This is the left-hand side of the equation \eqref{ODE1}, while the right-hand side is:
\begin{align}
\nonumber P+Qf(t)+R(f(t))^2&=P+Q\sum_{m=0}^{\infty}a_mt^{m\alpha}+R\sum_{k=0}^{\infty}\sum_{l=0}^{\infty}a_ka_lt^{(k+l)\alpha} \\
\label{ODE1:RHS} &=\sum_{m=0}^{\infty}\left(P\delta_{m0}+Qa_m+R\sum_{k=0}^ma_ka_{m-k}\right)t^{m\alpha}.
\end{align}
Equating coefficients in \eqref{ODE1:LHS} and \eqref{ODE1:RHS}, we find for $m=0$ that \[a_0\left(\frac{1-\alpha}{B(\alpha)}\right)^{\beta}=P+Qa_0+Ra_0^2\] and therefore
\begin{equation}
\label{ODE1:a0}
a_0=\frac{\left(\frac{1-\alpha}{B(\alpha)}\right)^{\beta}-Q\pm\sqrt{\left[\left(\frac{1-\alpha}{B(\alpha)}\right)^{\beta}-Q\right]^2-4PR}}{2R},
\end{equation}
while for $m>0$ we have
\begin{multline*}
\frac{1}{\Gamma(m\alpha+1)B(\alpha)^{\beta}}\left[a_m(1-\alpha)^{\beta}\Gamma(m\alpha+1)+\sum_{k=1}^ma_{m-k}\binom{\beta}{k}(1-\alpha)^{\beta-k}\alpha^k\Gamma((m-k)\alpha+1)\right] \\ =Qa_m+R\left[2a_ma_0+\sum_{k=1}^{m-1}a_ka_{m-k}\right]
\end{multline*}
and therefore \[a_m\left[\left(\frac{1-\alpha}{B(\alpha)}\right)^{\beta}-Q-2Ra_0\right]=R\sum_{k=1}^{m-1}a_ka_{m-k}-\frac{\sum_{k=1}^ma_{m-k}\binom{\beta}{k}(1-\alpha)^{\beta-k}\alpha^k\Gamma((m-k)\alpha+1)}{\Gamma(m\alpha+1)B(\alpha)^{\beta}}.\]
And the formula \eqref{ODE1:a0} for $a_0$ enables us to simplify the $a_m$ coefficient here. Thus, we derive the following general expression for the solution $f(t)$ of \eqref{ODE1}:
\begin{equation}
\label{ODE1:soln}
f(t)=a_0+\sum_{m=1}^{\infty}\left[\frac{R\sum_{k=1}^{m-1}a_ka_{m-k}-\frac{\sum_{k=1}^ma_{m-k}\binom{\beta}{k}(1-\alpha)^{\beta-k}\alpha^k\Gamma((m-k)\alpha+1)}{\Gamma(m\alpha+1)B(\alpha)^{\beta}}}{\mp\left(\left[\left(\tfrac{1-\alpha}{B(\alpha)}\right)^{\beta}-Q\right]^2-4PR\right)^{1/2}}\right]t^{m\alpha},
\end{equation}
where the constant term $a_0$ is given by \eqref{ODE1:a0}.

Differential equations of this and similar forms are important in the modelling of various real-world problems. As an example application, motivated by the physical study of \cite{sun}, we consider the following ODE, which can be used to create a variable-order system and model relaxation processes for a fractor in an electronic circuit.
\begin{equation}
\label{ODE2}
\mathcal{D}_{0+}^{(\alpha,\beta)}f(t)=-Cf(t)+q(t),
\end{equation}
where $q(t)$ is a known forcing function, $C$ is a constant, and $\mathcal{D}$ is defined by \eqref{IAB:defn:deriv} with $0<\alpha<1$ and $\beta>0$. We assume that $q$ can be written in the form \[q(t)=\sum_{n=0}^{\infty}c_nt^{n\alpha},\] and again we use \eqref{ansatz} as our ansatz for the solution $f$.

By the same approach as we used to derive \eqref{ODE1:LHS} and \eqref{ODE1:RHS}, we find that \eqref{ODE2} is equivalent to the identity
\begin{equation}
\label{ODE2:coeffs}
\frac{1}{\Gamma(m\alpha+1)}\sum_{k=0}^m\frac{a_{m-k}\binom{-\beta}{k}(1-\alpha)^{-\beta-k}\alpha^k\Gamma((m-k)\alpha+1)}{B(\alpha)^{-\beta}}=-Ca_m+c_m,
\end{equation}
valid for all $m\geq0$. Solving this for $m=0$, we find
\begin{equation}
\label{ODE2:a0}
a_0=\frac{c_0}{C+\big(\frac{B(\alpha)}{1-\alpha}\big)^{\beta}},
\end{equation}
while for $m>0$ the identity \eqref{ODE2:coeffs} rearranges to
\begin{equation}
\label{ODE2:am}
a_m=\frac{c_m}{C+\big(\frac{B(\alpha)}{1-\alpha}\big)^{\beta}}-\sum_{k=1}^m\frac{a_{m-k}\binom{-\beta}{k}\alpha^kB(\alpha)^{\beta}\Gamma((m-k)\alpha+1)}{(1-\alpha)^{\beta+k}\Gamma(m\alpha+1)\left(C+\big(\frac{B(\alpha)}{1-\alpha}\big)^{\beta}\right)}
\end{equation}
Substituting \eqref{ODE2:a0} and \eqref{ODE2:am} into the ansatz \eqref{ansatz}, we find a solution to the ODE \eqref{ODE2} in the following form:
\begin{equation}
\label{ODE2:soln}
f(t)=\frac{q(t)}{C+\big(\frac{B(\alpha)}{1-\alpha}\big)^{\beta}}-\sum_{m=1}^{\infty}t^{m\alpha}\sum_{k=1}^m\frac{a_{m-k}\binom{-\beta}{k}\alpha^kB(\alpha)^{\beta}\Gamma((m-k)\alpha+1)}{(1-\alpha)^{\beta+k}\Gamma(m\alpha+1)\left(C+\big(\frac{B(\alpha)}{1-\alpha}\big)^{\beta}\right)}
\end{equation}

As discussed in \cite{sun}, this solution to \eqref{ODE2} can be used to predict the behaviour of a dynamic-order fractional dynamic system, which is a way of modelling certain properties of a fractor in an electronic circuit.

\section{Conclusions}

In this manuscript we have introduced a new type of fractional calculus by iterating the integral corresponding to the fractional differintegral with Mittag-Leffler kernel. Our model falls into the Riemann--Liouville class, being defined by an integral formula with a special function in the kernel, and it contains two parameters in the order of differintegration. We believe that having a non-local operator with two parameters will enable us to describe better the non-local behaviour of the dynamics of complex systems in this new model.

We have proved some important fundamental properties of our new type of fractional calculus: evaluating their Laplace transforms, establishing a semigroup property -- which is significant in any fractional model -- and proving the boundedness of the new operators. We also presented some concrete applications, solving some related fractional differential equations and indicating how these can be applied to certain real-world systems.

\end{document}